\documentclass{article}

\usepackage{xcolor}
\usepackage{parskip}
\usepackage{graphicx}
\usepackage{hyperref}
\usepackage{algorithm}
\usepackage{algorithmic}
\usepackage{amsmath,amssymb, amsthm}
\usepackage[a4paper, total={6in, 8in}]{geometry}
\usepackage[backend=biber,
            style=authoryear,       
            citestyle=authoryear,   
            natbib=true,            
            maxnames=2,              
            minnames=1,
            maxbibnames=99,         
            dashed=false,           
            uniquename=false,       
            uniquelist=true,
            labeldateparts=true,         
            doi=false,                
            isbn=false,              
            url=false,                
            bibencoding=utf8        
           ]{biblatex}

\newcommand{\wrt}{wrt.\@}

\newcommand{\rhs}{right hand side}
\newcommand{\lhs}{left hand side}
\newcommand{\T}{^{\top}}
\newcommand{\inv}{^{-1}}
\newcommand{\iid}{i.i.d.\@}
\newcommand{\st}{s.t.\@}

\newcommand{\twonorm}[1]{\|#1\|_{2}}

\newcommand{\opnorm}[1]{\|#1\|_{op}}
\newcommand{\frobnorm}[1]{\|#1\|_{F}}
\newcommand{\abs}[1]{|#1|}

\newcommand{\Abs}[1]{\left|#1\right|}

\newcommand{\E}{\mathbb{E}}

\newcommand{\var}[2][]{\operatorname{Var}_{#1}(#2)}

\newcommand{\cov}[2][]{\operatorname{Cov}_{#1}(#2)}
\newcommand{\cor}[2][]{\operatorname{Cor}_{#1}(#2)}
\newcommand{\ent}[2][]{\operatorname{Ent}_{#1}(#2)}

\newcommand{\indicator}[1]{\mathbf{1}{\{#1\}}}

\newcommand{\DKL}[2]{{\operatorname{D_{KL}}(#1||#2)}} 

\renewcommand{\P}{\mathbb{P}} 

\newcommand{\cardinality}[1]{|#1|}
\newcommand{\R}{\mathbb{R}}
\newcommand{\N}{\mathbb{N}}

\newcommand{\diagonal}[1]{\operatorname{diag}(#1)}
\newcommand{\eigmax}[1]{\lambda_{\max}{(#1)}}
\newcommand{\eigmin}[1]{\lambda_{\min}{(#1)}}

\newcommand{\one}{\mathbf{1}}

\newcommand{\bracket}[1]{[#1]}
\newcommand{\curly}[1]{\{{#1}\}}

\newcommand{\Parent}[1]{\left(#1\right)}
\newcommand{\Bracket}[1]{\left[#1\right]}

\theoremstyle{plain}

\newtheorem{definition}{Definition}[section]
\newtheorem{lemma}[definition]{Lemma}
\newtheorem{theorem}[definition]{Theorem}
\newtheorem{corollary}[definition]{Corollary}
\newtheorem{auxiliary}[definition]{Auxiliary Result}
\newtheorem{assumption}[definition]{Assumption}

\theoremstyle{definition}

\newtheorem*{remark}{Remark}

\title{On McDiarmid's Inequality under Dependence \\ via Approximate Tensorization of Entropy}
\author{Valentin Roth}
\date{May 2026} 

\begin{document}

\maketitle

\begin{abstract}

    We argue that dependent versions of McDiarmid's inequality are a useful but underutilized tool in mathematical statistics, learning theory and theoretical computer science. To make this point, we first highlight that approximate tensorization of entropy (ATE) implies McDiarmid's via the Entropy Method. Second, we derive McDiarmid's inequality for non-isotropic Gaussian random vectors $X \sim \mathcal N(\mu, \Sigma)$ through ATE with a constant of the order of the condition number of $\Sigma$. We both independently obtain this ATE through a simple application of stochastic localization and also discuss how a more general ATE for the Gibbs sampler due to \cite{ascolani:lavenant:zanella2026} generalizes McDiarmid's-like concentration to strongly log-concave and log-smooth probability measures. We then apply the resulting concentration inequalities to resolve a question on the concentration of $\operatorname{sign}(X)$ posed by Simone Bombari, investigate Erdős-Rényi graphs under dependence and prove a Dvoretzky-Kiefer-Wolfowitz-type inequality for observations from a joint measure fulfilling ATE and continuous marginal CDFs. For the class of strongly log-concave and log-smooth measures, this result improves upon a prior Dvoretzky-Kiefer-Wolfowitz-type inequality for non-i.i.d.\ observations due to \cite{bobkov:goetze2010}, by establishing the expected $1/\sqrt{n}$-rate of convergence under weak dependence instead of $n^{-1/3}$. 

\end{abstract}

\setcounter{tocdepth}{2} 
\tableofcontents

\section{Introduction}

Concentration of measure is at the heart of statistics, learning theory and theoretical computer science, most often used in form of Chernoff-type concentration inequalities for sub-Gaussian and sub-exponential random variables (See e.g., \cite{motwani:raghavan1995, vandervaart:wellner1996, mohri:rostamizadeh:talwalkar2018, bach2024}). Such concentration inequalities are an integral part of the toolbox of modern theoretical statisticians and computer scientists and were popularized in the machine learning community through monographs by \cite{boucheron:lugosi:massart2013, vershynin2018, wainwright2019}. Their essence is summarized in the following quote due to \cite{talagrand1996}: 
\begin{center}
    \textit{"A random variable that depends (in a 'smooth' way) on the influence of many \\ independent random variables (but not too much on any of them) is essentially constant"}.
\end{center}

In this work we highlight that the statement above also extends to dependent random variables, as long as dependence is not too strong. One well-known example of this is Gaussian and more generally log-Sobolev Lipschitz concentration (\cite[Theorem 5.3]{ledoux2001}). It allows us to concentrate Lipschitz functions of Gaussian random vectors $X \sim \mathcal N(\mu, \Sigma)$ where $\Sigma \in \R^{n\times n}$ does not have to be the identity. Hence, its entries can be non-\iid{} Dependence between the entries of $X$ is thereby captured in the covariance matrix $\Sigma$ and enters the concentration inequality in form of its largest eigenvalue $\opnorm{\Sigma}$. This suffices to establish fundamental results such as the concentration of an empirical mean of the marginals $\bar X_n := \frac 1 n \sum_{i=1}^n X_i$ around its true mean $\E[\bar X_n]$ and thereby allows us to characterize the change of the rate of convergence depending on the strength of dependence. 

However, Gaussian Lipschitz concentration fails to cover many simple tasks such as concentrating sums of indicators of dependent Gaussian random variables, due to the failing Lipschitzness \wrt{} Euclidean distance. In the independent setting, this can be easily achieved through an application of Hoeffding's inequality that provides sub-Gaussian tails for sums of bounded random variables (\cite{hoeffding1963}). McDiarmid's inequality generalizes the behavior of Hoeffding's inequality and provides sub-Gaussian tails for "smooth" functions of independent random variables in the sense of Talagrand's quote (\cite{mcdiarmid1989, mcdiarmid1998}). The inequality is also known as the bounded difference inequality, which stems from the following "smoothness" assumption it imposes on the function $f : \mathcal X^n \to \R$ to be concentrated. 

\begin{definition}{Bounded Differences Property}

    A function $f : \mathcal X^n \to \R$ fulfills the bounded differences property if for $c_1,...,c_n \geq 0$ and all $i \in [n]$: 
    \begin{align*}
        \sup_{y, x_1,...,x_n \in \mathcal X} |f(x_1,...,x_n) - f(x_1,...x_{i-1} y, x_{i+1},...,x_n)| \leq c_i.
    \end{align*}
\end{definition}

Concretely, this condition implies that $f$ is Lipschitz \wrt{} the Hamming distance $d_H(x,y) := \sum_{i=1}^n \indicator{x_i \neq y_i}$ with Lipschitz constant $L := \max_{i\in[n]} c_i$. Likewise, if $f$ is $L$-Hamming-Lipschitz, it fulfills bounded differences with constants $L$. This type of Lipschitzness is more suitable to the functions of interest in many problems of combinatorial nature, a collection of which are outlined in \cite{mcdiarmid1998}. The bounded difference property of $f$ and independence between random variables is then enough to obtain the following sub-Gaussian tail bound for $f$ around its mean. 

\begin{theorem}{McDiarmid's Inequality (\cite{mcdiarmid1989})}
    \label{thm:McDiarmids}

    Let $X \in \mathcal X^n$ have independent entries. Assume $f : \mathcal X^n \to \R$ fulfills the bounded differences property with constants $c := (c_1,...,c_n)\T$. Then, for all $t > 0$ we have
    \begin{align*}
        \P\Parent{|f(X) - \E[f(X)]| \geq t} 
        &\leq 2\exp\Parent{-\frac{2t^2}{\twonorm{c}^2}}. 
    \end{align*}
\end{theorem}

Given that McDiarmid's is applicable in cases where Gaussian Lipschitz concentration fails to apply, but suffers from the restriction to independent random variables, it is now natural to ask: 
\begin{center}
    \textit{Is there a simple tool that let's us apply \\ Hamming-Lipschitz concentration under dependence?}
\end{center}

Surely, when the dependence structure is known and only a few of the entries depend on each other, we would expect a similar behavior as in Theorem \ref{thm:McDiarmids} as we could treat dependent observations together and adjust the effective sample-size accordingly. This idea is formalized by a line of work that encodes information on the dependence structure between entries in dependency graphs, partitions entries into independent sets and then uses the (fractional) chromatic number to capture the strength of dependence. \cite{janson2004} was the first to derive a Hoeffding-type concentration inequality for sums of dependent random variables using this technique, followed by \cite{usunier:amini:gallanari2005, zhang:liu:wang:wang2019} that extended it to functions with the bounded difference property. 

Even earlier \cite{azuma1967} proved the Azuma-Hoeffding inequality, a generalization of Hoeffding's inequality, which holds for martingale differences instead of the special case of sums of independent random variables and is now termed the Martingale Method. This approach of capturing dependence through martingales and filtrations was later extended to bounded differences (\cite{vandegeer2002}) and \cite{kantorovich:ramanan2008} exploit it to obtain a McDiarmid's-like concentration inequality for Hamming-Lipschitz functions on discrete state spaces. When combined with Wasserstein matrices \cite{kontorovich:raginsky2017} show that the martingale method can unify some results obtained through other approaches to concentration with bounded differences, e.g., via the Dobrushin uniqueness condition (\cite{djellout:guillin:wu2004, wu2006, chazottes:collet:kuelske2007, wang:wu2014, paulin2014}) or couplings (\cite{kantorovich:ramanan2008}). 

An alternative to dependency graphs and the martingale method is to capture dependence using information theoretic concepts. For example, \cite{esposito:mondelli2023} establish McDiarmid's under dependence using Hellinger integrals. More widely used, however, is the entropy method, which derives McDiarmid's-type inequalities relying on so called sub-additivity or tensorization of entropy (See e.g. \cite[Theorem 6.2]{boucheron:lugosi:massart2013}). This notion already appearing in \cite{gross1975} holds for product measures, hence \iid{} random variables. The proof approach extends to dependent settings through approximate tensorization of entropy (ATE). Similar to log-Sobolev Lipschitz concentration where the log-Sobolev constant $\rho \geq 1$ captures the strength of dependence, ATE captures dependence in a constant $\kappa \geq 1$ that is 1 only for product measures. Katalin Marton pioneered this relaxation and established it first for Euclidean spaces (\cite{marton2013}). Thereafter, \cite{marton2015} and \cite{caputo:menz:tetali2015} derived ATE under versions of Dobrushin's uniqueness condition on discrete product spaces and \cite{goetze:sambale:sinulis2019} provided concentration inequalities under higher-order bounded differences for discrete product spaces that exploit ATE. Since, ATE has become a central notion in the study of mixing times of the Glauber and heat-bath block dynamics also known as Gibbs samplers on discrete spin glasses as it has direct implications on their mixing time (See e.g., \cite{caputo2022, blanca:caputo:chen:parisi:stefankovic2022, anari:koehler:vuong2024} for proofs of ATE through spectral and entropic independence). Further, \cite{eldan:shamir2022} and \cite{eldan:koehler:zeitouni2022} show Hamming-Lipschitz concentration and approximate variance tensorization for spin glasses through stochastic localization. These techniques are generalized and refined in \cite{chen:eldan2022} and \cite{anari:jain:koehler:pham:vuong2024}. While the study of dependence in spin glasses is interesting, it does not straightforwardly relate to the study of dependence in more classical problems in estimation and learning theory.

Only recently, works on ATE by \cite{caputo:salez2026} and \cite{ascolani:lavenant:zanella2026} again focused on the Euclidean case first considered by \cite{marton2013}. Especially the results of \cite{ascolani:lavenant:zanella2026} are remarkable as they imply ATE for Gaussian and strongly log-concave and log-smooth measures. Through McDiarmid's inequality under ATE this provides a versatile analogue to Gaussian and log-Sobolev Lipschitz concentration for the Hamming-Lipschitz case: 
\begin{center}
    \textit{Gaussian Hamming-Lipschitz concentration}. 
\end{center}

This manuscript aims to provide a comprehensive derivation of this tool that is mostly implicit and hidden between the parts of the literature on Markov chain mixing and concentration inequalities\footnote{\cite[Section 6]{kontorovich:raginsky2017} and \cite[Fact 3.5]{chen:liu:vigoda2021} outline that ATE generally implies McDiarmid's-like inequalities. However, they do not provide useful ATEs for the continuous setting.}. 

\subsection{Organization}

In Section \ref{sec:mcdiarmid} we first derive McDiarmid's inequality using ATE. In Section \ref{sec:ATE} we then provide an independent derivation of ATE for $\mathcal N(\mu, \Sigma)$ with $\kappa$ of the order of the condition number of $\Sigma$ through stochastic localization. Moreover, we discuss how the results by \cite{ascolani:lavenant:zanella2026} sharpen this ATE constant and establish ATE for a class of probability measures including all strongly log-concave and log-smooth ones. Section \ref{sec:applications} contains three increasingly involved examples of applications of McDiarmid's under ATE: a short proof that $\operatorname{sign}(X)$ is dimension-free sub-Gaussian, a simple example of how McDiarmid's reveals the effect of dependence in Erdős-Rényi graphs and a Dvoretzky-Kiefer-Wolfowitz-type inequality under ATE. We consider this inequality a \textit{main contribution} of this work as it is the first to establish the expected $1/\sqrt{n}$-rate of convergence under ATE, in particular for weakly dependent Gaussians. Finally, Section \ref{sec:discussion} discusses our work. 

\subsection{Notation}

Let $[n] := \curly{1,2,...,n}$, $\indicator{A}$ be the indicator function for a set $A$ and $A \sqcup B$ be the disjoint union between sets. Let $\one_d \in \R^d$ be the all ones vector and $I_d \in \R^{d \times d}$ the identity matrix. For matrices $A,B \in \R^{d\times d}$ denote their positive definite and positive semi-definite order as $A\prec B$ and $A\preceq B$ and operator norm as $\opnorm{A}$. For $x,y\in\mathcal{X}^n$, their Hamming distance is $d_H(x,y) := \sum_{i=1}^n \indicator{x_i \neq y_i}$ and $x_{-i} := (x_1,...,x_{i-1}, x_{i+1},...,x_n)\T$ denotes all but the $i$-th entry of $x \in \mathcal X^n$. For a probability measure $\nu$ that has a density \wrt{} Lebesgue measure $dx$ we abuse notation and write $\nu(x)$ for its density. A measure on $\R^d$ is $\alpha$-strongly log-concave and $\beta$-log-smooth if $\nu(x) \propto \exp(-V(x))$ is \st{} $\alpha I_d \preceq \nabla V(x) \preceq \beta I_d$ with $0 < \alpha \leq \beta < \infty$ for all $x \in \R^d$. Let $\nu_i$ and $\nu(\cdot|x_{-i})$ be its marginal and conditional measures. For two sequences $(a_n)_{n\in\N}$ and $(b_n)_{n\in\N}$ we write $b_n = O(a_n)$ or $a_n \lesssim b_n$ if $\exists C > 0$ and $n_0 \in \N$ \st{} $\forall n \geq n_0, |\frac{b_n}{a_n}| \leq C$. We write $a_n\asymp b_n$ if $a_n\lesssim b_n$ and $b_n\lesssim a_n$.

\section{McDiarmid's Inequality}
\label{sec:mcdiarmid}

\subsection{McDiarmid's via Entropy Method}

Concentration via the entropy method is discussed in detail in Chapter 6 of \cite{boucheron:lugosi:massart2013}. Alternatively, Chapter 3 in \cite{raginsky:igal2013} also outlines the proof idea behind the Entropy Method from a more information theoretic viewpoint. The key inequality behind the Entropy Method for McDiarmid's is tensorization of entropy, where entropy is defined as follows. 

\begin{definition}{Entropy}

    Let $\nu$ be a probability measure on $\R^n$. For an integrable $f : \R^n \to [0, \infty)$ \st{} $\int_{\R^n} f\abs{\log(f)}d\nu < \infty$, 
    \begin{align*}
        \ent[\nu]{f} := \int_{\R^n} f \log(f) d\nu - \Parent{\int_{\R^n} f d\nu}\log\Parent{\int_{\R^n} fd\nu}. 
    \end{align*}

\end{definition}

Whenever we use entropy the in the following, we assume that $f$ is \st{} the relevant quantities exist. For random variables $X \in \mathcal X^n$ with \iid{} entries, entropy tensorization refers to the concept that the entropy is sub-additive in the sense that it is upper bounded by the sum of the expectations of the conditional entropies, i.e., the entropies \wrt{} the conditional measure $\nu(\cdot|X_{i-1})$ on $\mathcal X_i$. 

\begin{lemma}{Entropy Tensorization (Lemma 4.1, \cite{caputo2022})}
    \label{lem:enttens}

    Let $\nu :=\nu_1 \otimes \dotsm \otimes \nu_n$ be a product probability measure on the product space $\mathcal X^n := (\mathcal X_1 \times \dotsm \times \mathcal X_n)$. Suppose that $X \sim \mathcal \nu$. Then, for all $f : \mathcal X^n \to [0,\infty)$,
    \begin{align*}
        \ent[\nu]{f}
        \leq \sum_{i=1}^n \E[\ent[\nu_i]{f}]
        = \sum_{i=1}^n \E[\ent[\nu(\cdot | X_{-i})]{f}].
    \end{align*}
\end{lemma}

Similarly, we formalize ATE in Definition \ref{def:ATE} below, where now a constant $\kappa \geq 1$ captures the deviation from the behavior of a product measure and thus bigger $\kappa$ indicates more dependence. We refer to the case where $\kappa \asymp 1$ as the case of weak dependence. 

\begin{definition}{Approximate Tensorization of Entropy}
    \label{def:ATE}

    A probability measure $\nu$ on a space $\mathcal X^n$ fulfills ATE with constant $\kappa \geq 1$ if for all $f : \mathcal X^n \to [0,\infty)$,
    \begin{align*}
        \ent[\nu]{f}
        \leq \kappa \cdot  \sum_{i=1}^n \E[\ent[\nu(\cdot | X_{-i})]{f}]. 
    \end{align*}
\end{definition}

Section \ref{sec:ATE} discusses how this definition can be established, e.g., for multivariate Gaussians. For now, we keep working with this abstract definition and establish McDiarmid's inequality for all measures fulfilling ATE. For the special case of $1$-ATE and product measures, these arguments recover Theorem \ref{thm:McDiarmids}. We want to reiterate that the proof hereafter is entirely known and based on the proof of \cite[Theorem 6.2]{boucheron:lugosi:massart2013}. 

\begin{theorem}{McDiarmid's Inequality under ATE}
    \label{thm:McDiarmidsATE}

    Let $X \sim \nu$ on $\mathcal X^n$ that fulfills $\kappa$-ATE. Assume $f : \mathcal X^n \to \R$ fulfills the bounded differences property with constants $c = (c_1,...,c_n)\T$. Then, for all $t > 0$ we have
    \begin{align*}
        \P\Parent{|f(X) - \E[f(X)]| \geq t} 
        &\leq 2\exp\Parent{-\frac{2t^2}{\kappa \twonorm{c}^2}}. 
    \end{align*}
\end{theorem}

\begin{remark}
    There are multiple extensions of Theorem \ref{thm:McDiarmids}, many of which should also extend to \ref{thm:McDiarmidsATE}, e.g., when bounded differences only holds with high probability (\cite{kutin2002, combes2024}) or the $c_i$ are functions of all entries but $x_i$, i.e., functions of $x_{-i}$ (\cite[Subsection 6.2]{boucheron:lugosi:massart2013}).
\end{remark}

\begin{proof}

    This proof combines ATE with Hoeffding's Lemma and the Herbst Argument outlined in Subsection \ref{sec:herbstargument}. Concretely, by $\kappa$-ATE and using the test function $\phi(x) := \exp(\lambda f(x))$ with $\lambda > 0$ we have the following line of arguments where the second inequality $(\star)$ remains to be shown: 
    \begin{align*}
        \ent[\nu]{\exp(\lambda f)} 
        &\leq \kappa \cdot \sum_{i=1}^n \E[\ent[\nu(\cdot | X_{-i})]{\exp(\lambda f)}]
        \overset{(\star)}\leq \kappa \cdot \E\Bracket{\sum_{i=1}^n \frac{c_i^2 \lambda^2}{8} \E\Bracket{\exp(\lambda f(X)) |X_{-i}}} \\
        &= \kappa \cdot \sum_{i=1}^n \frac{c_i^2 \lambda^2}{8} \E\Bracket{\exp(\lambda f(X))} 
        = \frac{\lambda^2 \kappa\twonorm{c}^2}{8} \E\Bracket{\exp(\lambda f(X))}. 
    \end{align*}

    The bound $\ent[\nu(\cdot | X_{-i})]{\exp(\lambda f)} \leq \frac{c_i^2\lambda^2}{8} \E[\exp(\lambda f(X))|X_{-i}]$ will hold by a specific form of Hoeffding's Lemma that is suitable for applying the Herbst Argument (See Corollary \ref{cor:hoeffdingsforherbst} and Lemma \ref{lem:herbstargument}). 
    
    As $f$ does not have conditional mean zero, we work with the centered $g_{X_{-i}}(x_i) := f(X_{-i}, x_i) - m_i$ where $m_i = \E[f(X_{-i}, X_i)|X_{-i}]$. Then, by the bounded differences property of $f$ it holds that
    \begin{align*}
        \sup_{x_i \in \R} g_{X_{-i}}(x_i) - \inf_{x_i^\prime \in \R} g_{X_{-i}}(x_i^\prime)
        = \sup_{x_i, x_i^\prime \in \R} |g_{X_{-i}}(x_i) - g_{X_{-i}}(x_i^\prime)|
        \leq c_i. 
    \end{align*}

    Above, $g_{X_{-i}}(x_i) \in [a_i,b_i]$ where $b_i-a_i = c_i$ and since $g_{X_{-i}}(x_i)$ also has zero mean, by Corollary \ref{cor:hoeffdingsforherbst}, 
    \begin{align*}
        \ent[\nu(\cdot|X_{-i})]{\exp(\lambda g_{X_{-i}})}
        &\leq \frac{c_i^2 \lambda^2}{8} \E[\exp(\lambda g_{X_{-i}}(X_i))| X_{-i}], \\
        \Leftrightarrow e^{-\lambda m_i} \cdot \ent[\nu(\cdot|X_{-i})]{\exp(\lambda f)}
        &\leq \frac{c_i^2 \lambda^2}{8} \E[\exp(\lambda f(X))| X_{-i}] \cdot e^{-\lambda m_i}. 
    \end{align*}

    Here, the second line holds by homogeneity of the entropy shown in Auxiliary Result \ref{aux:homoent}. Dividing by $e^{-\lambda m_i}$ thus justifies the main inequality $(\star)$. Now, the Herbst Argument of Lemma \ref{lem:herbstargument} is applicable and for all $\lambda > 0$ the following bound on the moment generating function of $f$ holds
    \begin{align*}
        \E[\exp(\lambda (f(X)-\E[f(X)]))] 
        \leq \exp\Parent{\frac{\lambda^2 \kappa \twonorm{c}^2/4}{2}}. 
    \end{align*}

    Accordingly, a standard Chernoff bound and the choice $\lambda = \frac{4t}{\kappa\twonorm{c}^2}$ recover the one-sided statement: 
    \begin{align*}
        \P\Parent{f(X) - \E[f(X)] \geq t}
        &\leq \exp(-\lambda t) \cdot \E[\exp\Parent{\lambda(f(X) - \E[f(X)])}] \\
        &\leq \exp(-\lambda t) \exp\Parent{\frac{\lambda^2 \kappa\twonorm{c}^2/4}{2}}
        = \exp\Parent{-\frac{t^2}{\frac{\kappa}{2} \twonorm{c}^2}}.
    \end{align*}

    The two-sided bound follows by noting that $-f$ also fulfills the bounded difference property. 
\end{proof}

\subsection{Hoeffding's Lemma and Herbst Argument}
\label{sec:herbstargument}

In this subsection we focus on the Herbst Argument at the core of the McDiarmid's proof in the previous section and show how its main condition can be established using Hoeffding's Lemma. 

\begin{lemma}{Herbst Argument (Proposition 2.14, \cite{massart2003})}
    \label{lem:herbstargument}

    Let $X \sim \nu$ where $\nu$ is a probability measure on $\R$. Suppose there exists $v \in (0, \infty)$ \st{} for all $\lambda > 0$ it holds that $\ent[\nu]{\exp(\lambda X)} \leq \frac{\lambda^2v}{2} \E[\exp(\lambda X)]$. Then, for every $\lambda > 0$ we have
    \begin{align*}
        \E[\exp(\lambda (X-\E[X)])] 
        \leq \exp\Parent{\frac{\lambda^2v}{2}}. 
    \end{align*}
\end{lemma}

The Herbst Argument translates control of the entropy of test functions $\phi(x) := \exp(\lambda x)$ for all $\lambda > 0$ into sub-Gaussian control on a random variable's moment generating function. For the proof of McDiarmid's the argument's condition can be established using Hoeffding's Lemma. 

\begin{lemma}{Hoeffding's Lemma (Lemma 2.2, \cite{boucheron:lugosi:massart2013})}
    \label{lem:hoeffdings}

    Let $X \sim \nu$ have zero mean and $X \in [a,b]$, $\nu$-almost surely. Then, $X$ is $\sigma^2$-sub-Gaussian with $\sigma^2 := \frac{(b-a)^2}{4}$ and for the cumulant generating function $\psi(\lambda) := \log(\E[\exp(\lambda X)])$ we have $\psi^{\prime\prime}(\lambda) \leq \sigma^2$. 
\end{lemma}

Hoeffding's Lemma has Hoeffding's inequality as an immediate consequence (See \cite[Subsection 2.6]{boucheron:lugosi:massart2013}). The following corollary of Hoeffding's Lemma reformulates it to be of the form needed to apply the Herbst Argument. While this route would be overly complicated to establish Hoeffding's inequality, it is needed to exploit the conditional boundedness provided by the bounded difference property assumed by McDiarmid's. 

\begin{corollary}{Hoeffding's Lemma for Herbst Argument (Page 166, \cite{boucheron:lugosi:massart2013})}
    \label{cor:hoeffdingsforherbst}

    Let $X \sim \nu$ have zero mean and $X \in [a,b]$, $\nu$-almost surely. Then, we have 
    \begin{align*}
        \ent[\nu]{\exp(\lambda X)}
        \leq \frac{(b-a)^2\lambda^2}{8} \E\Bracket{\exp(\lambda X)}. 
    \end{align*}
    
\end{corollary}

\begin{proof}
    In the argument to follow, we show the equality below, where the subsequent inequality is a direct application of Hoeffding's Lemma \ref{lem:hoeffdings} stating that $\psi^{\prime\prime}(\lambda) \leq \frac{(b-a)^2}{4}$: 
    \begin{align*}
        \lambda \psi^\prime(\lambda) - \psi(\lambda)
        \overset{(\star)}= \int_0^\lambda \theta \psi^{\prime\prime}(\theta) d\theta
        \leq \frac{(b-a)^2\lambda^2}{8}. 
    \end{align*}

    Using the definition of $\psi^\prime(\lambda)$ expanding the \lhs{} above yields
    \begin{align*}
        \lambda \psi^\prime(\lambda) - \psi(\lambda)
        &= \lambda \Parent{\log(\E[\exp(\lambda X))]}^\prime - \log(\E[\exp(\lambda X)])
        = \frac{\E[\lambda X \exp(\lambda X)]}{\E[\exp(\lambda X)]} - \log(\E[\exp(\lambda X)]) \\
        &= \frac{1}{\E[\exp(\lambda X)]} \underbrace{\Parent{\E[\log(\exp(\lambda X)) \exp(\lambda X)] - \E[\exp(\lambda X)] \log(\E[\exp(\lambda X)])}}_{\ent[\nu]{\exp(\lambda X)}}.
    \end{align*}

    Using this form of the \lhs{} recovers the statement. Hence, it only remains to show that $(\star)$ is indeed valid. To do so, we define $F(\lambda) := \lambda \psi^\prime (\lambda) - \psi(\lambda)$. Now, 
    \begin{align*}
        F^\prime(\lambda) 
        = \frac{d}{d\lambda} (\lambda \psi^\prime(\lambda)) - \frac{d}{d\lambda} \psi(\lambda)
        = \psi^\prime(\lambda) + \lambda \psi^{\prime\prime}(\lambda) - \psi^\prime(\lambda)
        = \lambda \psi^{\prime\prime}(\lambda). 
    \end{align*}

    The Fundamental Theorem of Calculus and the fact that $F(0) = 0$ as $\psi(0) = \log(\E[\exp(0 X)]) = 0$ now recover the missing equality $(\star)$: 
    \begin{align*}
        \int_0^\lambda \theta \psi^{\prime\prime}(\theta) d\theta
        = F(\lambda) - F(0) 
        = \lambda \psi^\prime(\lambda) - \psi(\lambda). 
    \end{align*}
\end{proof}

\section{Approximate Tensorization of Entropy}
\label{sec:ATE}

In this section we first prove ATE for well-conditioned Gaussians and afterwards discuss a generalization and tightening of this result due to \cite{ascolani:lavenant:zanella2026}. 

Besides its importance for McDiarmid's, ATE is particularly central in the study of Gibbs samplers in Markov chain mixing (See e.g., \cite{caputo:menz:tetali2015, caputo2022, blanca:caputo:chen:parisi:stefankovic2022, chen:eldan2022, anari:jain:koehler:pham:vuong2024, anari:koehler:vuong2024}), as it implies that $O(\kappa n \log(1/\varepsilon))$ steps suffice to reduce the total variation distance and Kullback-Leibler divergence between the sample and stationary measure below $\varepsilon > 0$ (\cite[Corollary 3.5]{ascolani:lavenant:zanella2026}). 

\subsection{ATE for Gaussians via Stochastic Localization}

Stochastic localization is a recent tool in the study of high-dimensional probability measures that was first introduced by \cite{eldan2013}. Thereafter, it led to multiple breakthroughs towards proving the KLS conjecture posed in \cite{kannan:lovasz:simonovits1995} (\cite{lee:vempala2017, chen2021, klartag:lehec2022, jambulapati:lee:yin:vempala2022, klartag2023}). Moreover, it proved to be a versatile tool in Markov chain mixing (\cite{chen:eldan2022, alaoui:montanari:sellke2022, anari:koehler:vuong2024, huang:montanari:pham2024}). For detailed overviews on stochastic localization and connections to related works, see the recent surveys by \cite{montanari2023, shi:tian:zhang2025}. 

We use stochastic localization to establish ATE for Gaussians $\mathcal N(\mu, \Sigma)$ on $\R^n$ with constant proportional to the condition number $\kappa(\Sigma)$. We present this proof both to provide an independent derivation of ATE for the relevant case of multivariate Gaussians and because it is the simplest application of stochastic localization known to the authors. It therefore yields insights into the technique in a simple setting where intuition can also be gained through more elementary perspectives. 

\subsubsection{The Stochastic Localization Process}

The central object in stochastic localization is the following stochastic localization process. 

\begin{definition}{Stochastic Localization Process}
    \label{def:stoclocproc}

    Let $\nu$ be a probability measure on $\R^n$ and $B_t$ a standard Brownian motion in $\R^n$ adapted to a filtration $(\mathcal F_t)_{t\geq0}$. Let $(C_t)_{t \geq 0}$ be an $\mathcal F_t$-measurable process where $C_t \in \R^{n\times n}$ and $C_t \succ 0$. Then, we call the following measure-valued stochastic process a stochastic localization process
    \begin{align*}
        d\nu_t(x) \propto F_t(x)d\nu(x),
    \end{align*}

    where $F_t$ solves the SDEs $F_0(x) = 1$ and $dF_t(x) = F_t(x)(x-\int_{\R^n} yd\nu_t(y))\T C_t dB_t$ for all $x \in \R^n$. 
\end{definition}

The behavior of the stochastic localization process is controlled by the process for the driving matrices $C_t \in \R^{n \times n}$. When restricting the process to a constant driving matrix $C_t = C \in \R^{n\times n}$, it's density \wrt{} to the measure of interest $\nu$ is given in the following result. 

\begin{theorem}{Explicit Construction of $\nu_t$ (Theorem 2, \cite{alaoui:montanari2022})}
    \label{thm:stoclocproc}
    
    Let $X_0 \sim \nu_0$, $C \in \R^n, C \succ 0$, $W_t$ be a standard Brownian motion and define  $y_t = tX_0 + C\inv W_t$. Then, there exists a standard Brownian motion $B_t$ adapted to the Filtration generated by $y_t$ \st{} the following Radon-Nikodym derivative is a solution $F_t(x)$ for the SDE Definition \ref{def:stoclocproc}: 
    \begin{align*}
        \frac{d\nu_t(x)}{d\nu}
        \propto \exp\Parent{-\frac{t}{2} x\T C^2 x + y_t \T C^2 x}. 
    \end{align*}
\end{theorem}

\subsubsection{Entropy Conservation under Entropic Stability}

The idea behind entropic stability is that the change in entropy over the evolution of a stochastic localization process can be tracked if the entropy is roughly stable over time $t \in [0, T]$. This will help us establish ATE for well-conditioned Gaussians, since with the right choice of the driving matrix $C$ we will be able to transform them into isotropic Gaussians for which 1-ATE holds. The following lemma thereby formalizes how entropy is conserved along the stochastic localization process. 

\begin{lemma}{Approximate Entropy Conservation (Proposition 39, \cite{chen:eldan2022})}
    \label{lem:entropyconservation}

    Let $\nu_t$ be a stochastic localization process as in Definition \ref{def:stoclocproc} with driving matrix $C_t$. Fix $T > 0$ and suppose that for all $t \in [0,T]$ the $\nu_t$ are $\epsilon_t$-entropically stable \wrt{} $\psi(x,y) := \frac 1 2 \twonorm{C_t(x-y)}^2$. Then, $\nu_t$ fulfills the following approximate entropy conservation
    \begin{align*}
        \ent[\nu]{f}
        \leq \exp\Parent{\int_0^T \epsilon_t dt} \cdot \E[\ent[\nu_T]{f}], 
        \quad \text{for all $f : \R^n \to [0,\infty)$}.
    \end{align*}
\end{lemma}

Here, entropic stability is formalized in Definition \ref{def:entropicstability} below. 

\begin{definition}{Entropic Stability (Definition 29, \cite{chen:eldan2022})}
    \label{def:entropicstability}

    Let $\nu$ be a probability measure on $\R^n$ and $\epsilon > 0$. For all $v \in \R^n$ let $\mathcal T_v\nu$ be its exponential tilt defined through $\frac{d\mathcal T_v \nu(x)}{d\nu} \propto \exp(v\T x)$. The measure $\nu$ fulfills $\epsilon$-entropic stability \wrt{} $\psi : \R^n \times \R^n \to (0, \infty)$ if 
    \begin{align*}
        \psi\Parent{\int_{\R^n} x d\mathcal T_v\nu(x), \int_{\R^n} x d\nu(x)}
        \leq \epsilon \operatorname{D}_{KL}(\mathcal T_v \nu || \nu),
        \quad \text{for all} \quad 
        v \in \R^n. 
    \end{align*}
    
\end{definition}

The question now arises how entropic stability can be established. Luckily, for the entropic stability \wrt{} $\psi(x,y) = \frac 1 2 \twonorm{C_t(x-y)}^2$ as required by Lemma \ref{lem:entropyconservation} this can be done via covariance bounds. 

\begin{lemma}{Entropic Stability from Covariance Bounds (Lemma 40, \cite{chen:eldan2022})}
    \label{lem:entstabfromcovbound}

    Let $\nu$ be a measure on $\R^n$ and $C,A \in R^{n \times n}$ fulfill $A,C \succ 0$. Suppose that $\cov[\mathcal T_v \nu]{X} \preceq A$ for all $v \in \R^n$. Then, $\nu$ is $\opnorm{CAC}$-entropically stable \wrt{} $\psi(x,y) = \frac 1 2 \twonorm{C(x-y)}^2$. 
\end{lemma}

It is now only an exercise to prove that the covariance of strongly log-concave measures is bounded. 

\begin{corollary}{Covariance Bound for Tilts of Strongly Log-Concave Measures}
    \label{cor:covboundlogconc}

    Let $\nu$ on $\R^n$ be $\alpha$-strongly log-concave. Then, for all $v \in \R^n$ the $\mathcal T_v \nu(x) \propto \exp(v\T x) d\nu(x)$ fulfill
    \begin{align*}
        \cov[\mathcal T_v \nu]{X} 
        \preceq 1/\alpha \cdot I_n. 
    \end{align*}
\end{corollary}

\begin{proof}

    As we assume $\nu$ has a density, by definition for all $v \in \R^n$ the exponential tilt is of the form 
    \begin{align*}
        \mathcal T_v \nu(x)
        \propto \exp(v\T x) \nu(x)
        \propto \exp(-V(x) + v\T x). 
    \end{align*}

    With $Z$ the normalizing constant, taking a logarithm and derivatives yields
    \begin{align*}
        \log(\mathcal T_v \nu(x)) &= - V(x) + v\T x + Z \\
        \nabla_x \log(\mathcal T_v \nu(x)) &= -\nabla_x V(x) + v \\
        \nabla_x^2 \log(\mathcal T_v \nu(x)) &= -\nabla_x^2 V(x)
    \end{align*}

    Now, we have $-\nabla_x^2 \log(\mathcal T_v \nu(x)) = \nabla_x^2 V(x) \succeq \alpha I_n$ and therefore the exponential tilt is also strongly log-concave. By the Bakry-Emery criterion in Theorem \ref{thm:bakryemery} we then have
    \begin{align*}
        \ent[\mathcal T_v \nu]{f^2}
        \leq \frac 2 \alpha \cdot \E[\twonorm{\nabla f(X)}^2],
        \quad \text{and} \quad 
        \var[\mathcal T_v \nu]{f}
        \leq \frac 1 \alpha \cdot \E[\twonorm{\nabla f(X)}^2]. 
    \end{align*}

    Choosing $f(x) = \theta\T x$ for any $\theta \in \R^n$ we obtain $\cov[\mathcal T_v \nu]{X} \preceq 1/\alpha \cdot I_n$.
\end{proof}

\subsubsection{ATE from Entropy Conservation}

It remains to connect ATE to entropy conservation. For this, we implement the idea that we take a measure of interest, transform it into a product measure where 1-ATE holds and conserve entropy along this evolution. Following the proof of \cite[Lemma 54]{anari:koehler:vuong2024} we obtain: 

\begin{corollary}{ATE from Entropy Conservation}
    \label{cor:ATEfromCons}

    Let $\nu_t$ be a stochastic localization process as in Definition \ref{def:stoclocproc}. Suppose that $\nu_t$ fulfills $\gamma_T$-approximate entropy conservation as in Lemma \ref{lem:entropyconservation} over times $ t \in [0,T]$ and $\nu_T$ is a product measure. Then, 
    \begin{align*}
        \ent[\nu]{f} 
        \leq \gamma_T \cdot \sum_{i=1}^n \E[\ent[\nu(\cdot | X_{-i})]{f}]. 
    \end{align*}
    
\end{corollary}

\begin{proof}

    The proof is by the following line of inequalities: 
    \begin{align*}
        \ent[\nu]{f} 
        &\leq \gamma_T \cdot \E[\ent[\nu_T]{f}] \\
        &\leq \gamma_T \cdot \E[\sum_{i=1}^n \E_{X_{-i} \sim \nu_T K}[\ent[\nu_T(\cdot | X_{-i})]{f}]] \\
        &\leq \gamma_T \cdot \sum_{i=1}^n \E_{X_{-i} \sim \nu K}[\ent[\nu(\cdot | X_{-i})]{f}]. 
    \end{align*}

    The first inequality holds by $\gamma_T$-approximate entropy conservation. The second by entropy tensorization of Lemma \ref{lem:enttens} that is applicable as $\nu_T$ is a product measure. The third inequality holds by Lemma \ref{lem:supermartent} where $K(x,A) = \indicator{x_{-i} \in A}$ and hence $\nu | K \triangleright x_{-i} = \nu(\cdot | x_{-i})$. 
\end{proof}

\begin{lemma}{Supermartingality of Entropy (Lemma 39, \cite{anari:koehler:vuong2024})}
    \label{lem:supermartent}

    Assume $\nu_t$ is a stochastic localization process as in Definition \ref{def:stoclocproc}. Let $K$ be any Markov kernel. Then, the stochastic process $t \mapsto \E_{y\sim \nu_t K}[\ent[\nu_t | K \triangleright y]{f}]$ is a supermartingale. Let $(\mathcal F_t)_{t\geq 0}$ be the filtration generated by the Brownian motion $B_t$ of the process. Equivalently, for all $0 \leq s < t$
    \begin{align*}
        \E[\E_{y\sim \nu_t K}[\ent[\nu_t | K \triangleright y]{f}] | \mathcal F_s] 
        \leq \E_{y\sim \nu_s K}[\ent[\nu_s | K \triangleright y]{f}].
    \end{align*}
\end{lemma}

ATE for well-conditioned Gaussians now follows by collecting all of these results. 

\begin{theorem}{ATE for Gaussians}
    \label{thm:ATEgauss}
    
    Let $\mathcal N$ be the multivariate Gaussian probability measure with mean $\mu \in \R^n$ and covariance matrix $\Sigma \in \R^{n\times n}$ \st{} $\Sigma \succ 0$. Then, with $X \sim \mathcal N$ the following ATE holds: 
    \begin{align*}
        \ent[\mathcal N]{f} 
        \lesssim \kappa(\Sigma) \cdot \sum_{i=1}^n \E[\ent[\mathcal N(\cdot | X_{-i})]{f}]. 
    \end{align*}
\end{theorem}

\begin{proof}

    This proof is inspired by the proof of \cite[Theorem 54]{anari:koehler:vuong2024} that derives ATE for SK-models with well-behaved interaction matrices in the high-temperature regime. 

    As entropy is invariant to shifts of the mean $\mu$, we can show the statement for a zero-mean Gaussian. Let $\nu$ be the $n$-dimensional Gaussian measure $\mathcal N(0,\Sigma)$. Hence, its density is 
    \begin{align*}
        \nu(x) \propto \exp\Parent{-\frac 1 2 x\T \Sigma\inv x}. 
    \end{align*}

    Using the explicit construction of a stochastic localization process in Theorem \ref{thm:stoclocproc} with $\nu_0 := \nu$, positive definite driving matrix $C$ and process $y_t = t X_0 + C\inv B_t$ yields the following Radon-Nikodym derivatives \wrt{} $\nu$ and Lebesgue measure 
    \begin{align*}
        \frac{d\nu_t(x)}{d\nu}
        \propto \exp\Parent{-\frac{t}{2} x\T C^2 x + y_t \T C^2 x},
        \quad 
        \nu_t(x) 
        \propto \exp\Parent{-\frac{1}{2} x\T (\Sigma\inv + tC^2) x + y_t\T C^2 x}. 
    \end{align*}

    We would like to choose $C \succ 0$ \st{} $\Sigma\inv + TC^2 = c I_n$ for some $c > 0$ as this would allow us to apply Corollary \ref{cor:ATEfromCons} because $\nu_T$ would be a product measure. For any $\varepsilon > 0$, we can obtain this with 
    \begin{align*}
        c := (1+\varepsilon) \eigmax{\Sigma\inv} > \eigmax{\Sigma\inv}
        \quad \text{and} \quad 
        C := \sqrt{1/T} \cdot (cI_n - \Sigma\inv)^{1/2} 
        \succ 0. 
    \end{align*}  

    With these choices $\nu_T$ then indeed is a product measure, since 
    \begin{align*}
        TC^2 
        = T/T \cdot (cI_n - \Sigma\inv)
        = (cI_n - \Sigma\inv).
    \end{align*}

    We further show that $\nu_t$ are strongly log-concave for all $t \in [0,T]$: 
    \begin{align*}
        -\nabla_x^2 \log(\nu_t(x))
        = \Sigma\inv + tC^2
        &= \Sigma\inv + t/T \cdot (cI_n - \Sigma\inv) \\
        &= (1-t/T) \cdot \Sigma\inv + t/T \cdot cI_n \\
        &\succeq ((1-t/T) \cdot \eigmin{\Sigma\inv} + t/T \cdot c) I_n \\ 
        &= (\eigmin{\Sigma\inv} + t/T \cdot (c - \eigmin{\Sigma\inv})I_n 
        =: \alpha_t I_n. 
    \end{align*}

    Now, by Corollary \ref{cor:covboundlogconc} we have $\cov[\mathcal T_v \nu]{X} \preceq 1/\alpha_t \cdot I_n$ and thus by Lemma \ref{lem:entstabfromcovbound} $\nu_t$ is $\epsilon_t$-entropically stable \wrt{} the function $\psi(x,y) = \frac 1 2 \twonorm{C(x-y)}^2$ where 
    \begin{align*}
        \epsilon_t 
        := \frac{1}{T\alpha_t} \opnorm{cI_n - \Sigma\inv}
        = \frac{c-\eigmin{\Sigma\inv}}{T\eigmin{\Sigma\inv} + t(c - \eigmin{\Sigma\inv})}
        = \frac{c-d}{Td + t(c-d)}.
    \end{align*}

    Here, $d: = \eigmin{\Sigma\inv}$. Entropic stability lets us apply Lemma \ref{lem:entropyconservation} by which we have $\E[\ent[\nu_T]{f}] \geq \exp(-\int_0^T \epsilon_t dt) \cdot \ent[\nu_0]{f}$. By Corollary \ref{cor:ATEfromCons} we translate this into the following ATE: 
    \begin{align*}
        \ent[\nu]{f} 
        &\leq \exp\Parent{\int_0^T \epsilon_t dt} \cdot \sum_{i=1}^n \E[\ent[\nu(\cdot | X_{-i})]{f}] \\
        &\leq \exp\Parent{\Bracket{\log(Td + t(c-d))}^T_0} \cdot \sum_{i=1}^n \E[\ent[\nu(\cdot | X_{-i})]{f}]
        \leq \frac{c}{d} \cdot \sum_{i=1}^n \E[\ent[\nu(\cdot | X_{-i})]{f}].
    \end{align*}

    The statement follows as $\eigmin{\Sigma\inv} = 1/\eigmax{\Sigma}$ and $\eigmax{\Sigma\inv} = 1/\eigmin{\Sigma}$ and thus 
    \begin{align*}
        \frac{c}{d} 
        = (1+\varepsilon) \cdot \frac{\eigmax{\Sigma\inv}}{\eigmin{\Sigma\inv}}
        = (1+\varepsilon) \cdot \frac{\eigmax{\Sigma}}{\eigmin{\Sigma}} 
        = (1+\varepsilon) \cdot \kappa(\Sigma).
    \end{align*}
\end{proof}

\subsection{ATE for Strongly Log-Concave Measures}

\cite{ascolani:lavenant:zanella2026} approach ATE through a variational characterization of the kernel of the Gibbs sampler in terms of Kullback-Leibler divergences. They exploit that the kernel of the Gibbs sampler is invariant under coordinate-wise transformations, which allows them to lower bound the decay of Kullback-Leibler divergence using carefully chosen transport maps that preserve the block structure a Gibbs sampler operates on. With this approach, they derive an ATE for a class of measures with densities $\nu(x) \propto \exp(-V(x))$ that is a superset of strongly log-concave and log-smooth measures and which is characterized by Assumption \ref{ass:strongcondition} placed on their potential $V : \R^n \to \R$. 

\begin{assumption}{Strong Condition}
    \label{ass:strongcondition}

    Let $V_m : \R^{n_m} \to \R$ be convex for all $m \in [M]$ and let $V : \R^n \to \R$ with $n = n_1 +\dotsm+ n_M$ be
    \begin{align*}
        V(x) 
        = V_0(x) + \sum_{m=1}^M V_m(x_m), 
    \end{align*}

    where $V_0 : \R^n \to \R$ is continuously differentiable and
    \begin{enumerate}
        \item $x_m \mapsto V_0(x_m, y_{-m})$ is $L_m$-smooth for all $y_{-m} \in \R^{n_m}$ and $m \in [M]$,
        \item $x \mapsto V_0(x) - \frac{\lambda^\star}{2} \sum_{m=1}^M L_m \twonorm{x_m}^2$ is convex. 
    \end{enumerate}

    Further, suppose that $\lambda^\star > 0$ and call $\kappa^\star = {\lambda^\star}\inv$ the coordinate-wise condition number of $V$.
            
\end{assumption}

As this assumption on the potential $V$ might be hard to verify, \cite{ascolani:lavenant:zanella2026} provide the following weaker condition under which Assumption \ref{ass:strongcondition} holds. 

\begin{lemma}{Weak Condition (Lemma 2.4, \cite{ascolani:lavenant:zanella2026})}
    \label{lem:weakcondition}
    
    Let $\nu$ on $\R^n$ be an $\alpha$-strongly log-concave and $\beta$-log-smooth probability measure with density $\nu(x) \propto \exp(-V(x))$. Then, $V : \R^n \to \R$ satisfies the conditions in Assumption \ref{ass:strongcondition} with $\kappa = \beta/\alpha \geq \kappa^\star \geq 1$.
\end{lemma}

Imposing Assumption \ref{ass:strongcondition} on the potential $V$ of the measure of interest $\nu$ on $\R^n$, they prove:

\begin{theorem}{ATE for Well-Conditioned Measures (Theorem 3.1, \cite{ascolani:lavenant:zanella2026})}
    \label{thm:generalATE}

    Let the probability measure $\nu$ on $\R^n$ have density $\nu(x) \propto \exp(-V(x))$ where $V : \R^n \to \R$ fulfills Assumption \ref{ass:strongcondition}. Then, with $X \sim \nu$ we have
    \begin{align*}
        \ent[\nu]{f} 
        \leq \kappa^\star \cdot \sum_{i=1}^n \E[\ent[\nu(\cdot | X_{-i})]{f}]. 
    \end{align*}
    
\end{theorem}

For Gaussian measures $\mathcal N(\mu, \Sigma)$ on $\R^n$, this general ATE implies that our previously derived $\kappa(\Sigma)$-ATE is not tight. Instead, computing the tight constant $\kappa^\star$ obtained through verifying Assumption \ref{ass:strongcondition} for such Gaussians yields the $\kappa^\star(\Sigma)$-ATE in Lemma \ref{lem:ATEtightgauss} below where $\kappa^\star(\Sigma) := \eigmax{D^{1/2} \Sigma D^{1/2}}$ and $D = \diagonal{\Sigma\inv}$. This $\kappa^\star(\Sigma)$ is the maximum eigenvalue of the precision standardized covariance matrix instead of the condition number and $\kappa^\star(\Sigma) \leq \kappa(\Sigma)$.

\begin{lemma}{Tight ATE for Gaussians (Lemma 3.10, \cite{ascolani:lavenant:zanella2026})}
    \label{lem:ATEtightgauss}

    Let $\mathcal N$ be the multivariate Gaussian probability measure with mean $\mu \in \R^n$ and covariance matrix $\Sigma \in \R^{n \times n}$ and let $X \sim \mathcal N$. Then, with $\kappa^\star(\Sigma) := \eigmax{D^{1/2} \Sigma D^{1/2}}$ where $D := \diagonal{\Sigma\inv}$ we have
    \begin{align*}
        \ent[\mathcal N]{f} 
        \leq \kappa^\star(\Sigma) \cdot \sum_{i=1}^n \E[\ent[\mathcal N(\cdot | X_{-i})]{f}].
    \end{align*}
        
\end{lemma}

While \cite{ascolani:lavenant:zanella2026} discuss the implications of their work on mixing times of the Gibbs sampler and their relation to coordinate-wise algorithms in optimization, they do not highlight the consequences for Hamming-Lipschitz concentration. Corollary \ref{cor:mcdiarmidsgaussian} below summarizes the important Gaussian case obtained by plugging the $\kappa^\star(\Sigma)$-ATE into Theorem \ref{thm:McDiarmidsATE}. This is the inequality we refer to as \textit{Gaussian Hamming-Lipschitz concentration} and which we believe to have many potential applications across mathematical statistics, learning theory and theoretical computer science. 

\begin{corollary}{Gaussian Hamming-Lipschitz Concentration}
    \label{cor:mcdiarmidsgaussian}

    Let $X \sim \mathcal N(\mu, \Sigma)$ on $\R^n$ with $\Sigma \in \R^{n \times n}$. Assume $f : \R^n \to \R$ fulfills the bounded differences property with constants $c = (c_1,...,c_n)\T$. Then, with $\kappa^\star(\Sigma) := \eigmax{D^{1/2} \Sigma D^{1/2}}$ where $D := \diagonal{\Sigma\inv}$, 
    \begin{align*}
        \P\Parent{|f(X) - \E[f(X)]| \geq t} 
        \leq 2\exp\Parent{-\frac{2t^2}{\kappa^\star(\Sigma) \twonorm{c}^2}}
        \quad \text{for all $t > 0$}. 
    \end{align*}
\end{corollary}

\begin{remark}
    Note again that $\kappa^\star(\Sigma) \leq \kappa(\Sigma)$ and equality holds only in the case when $\Sigma = cI_n$ for $c > 0$. Importantly, $\kappa^\star(\Sigma)$ can indeed be infinitely smaller as can be seen in the following example: 
    \begin{align*}
        \Sigma 
        := \begin{pmatrix}
            s & 0 \\
            0 & 1
        \end{pmatrix}
        \quad \text{where $s > 1$}. 
    \end{align*}

    Then, $\kappa(\Sigma) = s$ but $\kappa^\star(\Sigma) = 1$ and hence as $s \to \infty$ the condition number diverges. 
\end{remark}

\section{Applications}
\label{sec:applications}

McDiarmid's inequality has many applications in combinatorial problems as presented in \cite{mcdiarmid1998}, \cite[Chapter 3.2]{boucheron:lugosi:massart2013} or \cite[Example 2.24]{wainwright2019}. Examples of classical statistical problems where the inequality can be applied are the concentration of U-statistics or the $L_1$-norm error of kernel density estimators (\cite[Example 2.23, Exercise 2.15]{wainwright2019}) as well as generalization bounds for cross-validation (See e.g., \cite{celisse:mary2018, lei2025}). In mathematical statistics and learning theory McDiarmid's is used to establish uniform laws of large numbers via Rademacher complexities as in \cite[Theorem 4.10]{wainwright2019} or generalization bounds through algorithmic stability (\cite{bousquet:elisseeff2002, rakhlin:mukherjee:poggio2005}. 

Here, we present three examples where McDiarmid's under dependence (Theorem \ref{thm:McDiarmidsATE}) can be applied for simple but very different problems. We thereby do not try to be exhaustive but want to provide examples that might spark the imagination of readers to come up with applications of Gaussian Hamming-Lipschitz concentration in their own fields of expertise. 

\subsection{Sub-Gaussianity of Sign-Quantized Vectors}

The following question was pitched to us by Simone Bombari: 

\begin{center}
    \textit{When does the random vector $\operatorname{sign}(X) \in \curly{\pm 1}^n$ where $X \sim \mathcal N(\mu, \Sigma)$ with $\mu \in \R^n$ and $\Sigma \in \R^{n\times n}$ fulfill dimension-free sub-Gaussian concentration in the sense of Definition \ref{def:subgaussianvector}?}
\end{center}

\begin{definition}{Sub-Gaussian Vector}
    \label{def:subgaussianvector}

    $X \in \R^n$ is a $\sigma^2$-sub-Gaussian vector if there exists $\sigma^2 > 0$ \st{}
    \begin{align*}
        \E\Bracket{\exp\Parent{u\T X - \E[u\T X]}}
        \leq \exp\Parent{\frac{\sigma^2\twonorm{u}^2}{2}} 
        \quad \text{for all $u \in \R^n$}. 
    \end{align*}
\end{definition}

Besides being an interesting exercise in high-dimensional probability, this question has applications in the analysis of deep neural networks via Neural Tangent Kernels (See e.g., \cite{bombari:amani:mohammad:mondelli2022}). 

When $\Sigma = I_n$ the question can easily be answered using McDiarmid's inequality, but if $\Sigma \not\asymp I_n$ the classical concentration inequalities including Gaussian Lipschitz concentration fail to apply directly. However, Theorem \ref{thm:McDiarmidsATE} immediately provides a strong answer for measures fulfilling ATE and thus well-conditioned multivariate Gaussians and strongly log-concave and log-smooth distributions. 

\begin{corollary}{Sub-Gaussianity of Sign-Quantized Vectors under ATE}

    Let $X \sim \nu$ and $\nu$ on $\R^n$ fulfill $\kappa$-ATE. Then, $\operatorname{sign}(X) \in \curly{\pm1}^n$ is $\kappa$-sub-Gaussian. 
    
\end{corollary}

\begin{remark}

    Recall that if $\nu$ is $\alpha$-strongly log-concave and $\beta$-log-smooth, Assumption \ref{ass:strongcondition} holds with $\kappa^\star \leq \kappa := \beta/\alpha$ by Lemma \ref{lem:weakcondition} and hence $\nu$ fulfills $\kappa$-ATE by Theorem \ref{thm:generalATE}. For a Gaussian $X \sim \mathcal N(\mu,\Sigma)$ this suggests that $\operatorname{sign}(X)$ is at least $\kappa(\Sigma)$-sub-Gaussian. Relying instead on Lemma \ref{lem:ATEtightgauss} improves this to $\kappa^\star(\Sigma)$-sub-Gaussianity with $\kappa^\star(\Sigma) := \eigmax{D^{1/2} \Sigma D^{1/2}}$ where $D := \diagonal{\Sigma\inv}$. 
    
\end{remark}

\begin{proof}

    We prove the statement by an application of McDiarmid's inequality. 

    Define $f(x) := u\T\operatorname{sign}(x)$ for all $u \in \R^n$. Then, for all $x, y \in \R^n$ with Hamming distance $d_H(x,y) \leq 1$ where the $k$-th entries differ, it holds that 
    \begin{align*}
        \Abs{f(x)-f(y)}
        &= \Abs{u\T (\operatorname{sign}(x)-\operatorname{sign}(y))}
        = \Abs{\sum_{i=1}^n u_i (\operatorname{sign}(x_i) - \operatorname{sign}(y_i))} \\
        &= \abs{u_k (\operatorname{sign}(x_k) - \operatorname{sign}(y_k))}
        \leq \abs{u_k} \abs{\operatorname{sign}(x_k) - \operatorname{sign}(y_k)}
        \leq 2 \abs{u_k} =: c_k.
    \end{align*}

    This means that $f$ fulfills the bounded difference condition with $\twonorm{c}^2 = 4 \twonorm{u}^2$. Theorem \ref{thm:McDiarmidsATE} yields
    \begin{align*}
        \E[\exp(\lambda (u\T \operatorname{sign}(X)-\E[u\T \operatorname{sign}(X)])] 
        &\leq \exp\Parent{\frac{\lambda^2 \kappa \twonorm{u}^2}{2}}
        \quad \text{for all $\lambda > 0$}. 
    \end{align*}

    Choosing $\lambda = 1$ recovers the statement.   
\end{proof}   

During the preparation of this manuscript \cite{zou:vershynin2026} provided an alternative proof for the sub-Gaussianity of $\operatorname{sign}(X)$ where $X \sim \mathcal N(0, \Sigma)$ with variance proxy of the order $O(\kappa(\Sigma))$. Their elementary proof decomposes $X$ into the sum of two Gaussians and exploits Gaussian Lipschitz concentration after smoothing the $\operatorname{sign}$-function with one of them. Hence, it is restricted to Gaussians and happens to be less tight in the constants that are involved. 

\subsection{Erdős-Rényi Graphs under Dependence}

Erdős-Rényi graphs (ERG) are a powerful model in graph theory and complex networks. An ERG is a graph with $n$ nodes where two edges are connected with success probability $p \in (0,1)$, independently. Many statistics of interest defined on such graphs fulfill variations of the bounded difference property and thus McDiarmid's inequality and versions thereof are a handy tool for their analysis. In the following, we show how ERGs can be parametrized by Gaussian instead of Bernoulli random variables, which let's us analyze them using Corollary \ref{cor:mcdiarmidsgaussian}. 

\begin{definition}{Dependent Erdős-Rényi Graph}

    Let $X \sim \mathcal N(\mu, \Sigma)$ with $\mu \in \R^{N}$ and $\Sigma \in \R^{N}$ where $N = \binom{n}{2}$. Index the mean $\mu$, covariance $\Sigma$ and random vector $X$ by tuples $(i, j) \in [n]^2$ where $i < j$. We define a dependent Erdős-Rényi graph as a graph $G_{\mu, \Sigma}(X) := (V, E(X))$ where $V = [n]$ and $(i,j) \in E(X)$ if $\indicator{X_{(i,j)} \geq 0}$ .
    
\end{definition}

\begin{corollary}{Control on Marginal Edge Probabilities}

    For $\mu^p := \Phi\inv(p) \cdot \diagonal{\Sigma}^{1/2}$ where $p \in (0,1)$ and $\Phi : \R \to [0,1]$ is the standard Gaussian CDF,
    \begin{align*}
        \P((i,j) \in E(X)) = p
        \quad \text{for all $(i,j) \in [n]^2$ where $i<j$}. 
    \end{align*}
    
\end{corollary}

\begin{proof}

    Since $Z \overset{d}{=} -Z$ for $Z \sim \mathcal N(0,1)$ and $\mu^p_{(i,j)} = \Phi\inv(p) \sqrt{\Sigma_{(i,j), (i,j)}}$ it holds that
    \begin{align*}
        \P((i,j) \in E(X))
        &= \P(X_{(i,j)} \geq 0) 
        = \P\Parent{\frac{X_{(i,j)} - \mu_{(i,j)}}{\sqrt{\Sigma_{(i,j), (i,j)}}} \geq - \frac{\mu_{(i,j)}}{\sqrt{\Sigma_{(i,j), (i,j)}}}} 
        = \P\Parent{Z \geq - \frac{\mu_{(i,j)}}{\sqrt{\Sigma_{(i,j), (i,j)}}}} \\
        &= \P\Parent{Z \leq \frac{\mu_{(i,j)}}{\sqrt{\Sigma_{(i,j), (i,j)}}}}
        = \Phi\Parent{\frac{\mu_{(i,j)}}{\sqrt{\Sigma_{(i,j), (i,j)}}}}
        = \Phi\Parent{\Phi\inv(p)}
        = p. 
    \end{align*}
    
\end{proof}

\begin{remark}

    $G_{\mu, \Sigma}(X)$ is a standard ERG with edge probability $p$ if $\mu = \mu^p$ and $\Sigma = I_N$. 
    
\end{remark}

The statistic we are interested in in this expository treatment of ERGs is the maximum cut below. 

\begin{definition}{Maximum Cut}

    For a graph $G = (V,E)$ and with $S \sqcup S^c = [n]$ we define the maximum cut as
    \begin{align*}
        \operatorname{MaxCut}(G) 
        := \max_{S \subseteq V} \operatorname{Cut}_S(G)
        = \max_{S \subseteq V} \sum_{i<j}^n \indicator{(i,j) \in E : \cardinality{\curly{i,j} \cap S} = 1}.
    \end{align*}
\end{definition}

Since the maximum cut fulfills the bounded difference property with $c_i = 1$, we immediately obtain a concentration inequality for the statistic around its expectation. 

\begin{lemma}{Concentration of the Maximum Cut}

    For a fraction $\epsilon \in (0,1)$ the maximum cut of the dependent ERG $G_{\mu^p, \Sigma}(X)$ fulfills 
    \begin{align*}
        \P\Parent{\Abs{\operatorname{MaxCut}(G_{\mu^p, \Sigma}(X))- \E[\operatorname{MaxCut}(G_{\mu^p, \Sigma}(X))]} \geq \epsilon \binom{n}{2}}
        \leq 2\exp\Parent{-\frac{2 \epsilon^2 \binom{n}{2}}{\kappa^\star(\Sigma)}}, 
    \end{align*}

    where $\kappa^\star(\Sigma) := \eigmax{D^{1/2} \Sigma D^{1/2}}$ with $D := \diagonal{\Sigma\inv}$ is the ATE constant for $\Sigma$. 
    
\end{lemma}

\begin{remark}

    If $\binom{n}{2}/\kappa^\star(\Sigma) \to \infty$ as $n \to \infty$ the maximum cut concentrates around its expectation. For $\kappa^\star(\Sigma) \asymp 1$ the rate of convergence is the same as for a standard ERG where $\Sigma = I_N$. 
    
\end{remark}

\begin{proof}

    For two graphs $G, G^\prime$ with edge sets $E, E^\prime \subseteq \curly{(i,j) \in [n]^2 : i < j}$ differing in edge $(u,v)$, 
    \begin{align*}
        \operatorname{MaxCut}(G) &- \operatorname{MaxCut}(G^\prime) \\
        &= \max_{S \subseteq V} \operatorname{Cut}_S(G) - \max_{S \subseteq V} \operatorname{Cut}_S(G^\prime)
        \leq \max_{S \subseteq V} \operatorname{Cut}_S(G) - \operatorname{Cut}_S(G^\prime) \\
        &= \max_{S \subseteq V} \sum_{i<j}^n \indicator{(i,j) \in E : \cardinality{\curly{i,j} \cap S} = 1} - \indicator{(i,j) \in E^\prime : \cardinality{\curly{i,j} \cap S} = 1} \\
        &= \indicator{(u,v) \in E : \cardinality{\curly{i,j} \cap S} = 1} - \indicator{(u,v) \in E^\prime : \cardinality{\curly{i,j} \cap S} = 1}
        \leq 1. 
    \end{align*}

    By symmetry the same holds with $G$ and $G^\prime$ swapped and thus $\operatorname{MaxCut}$ fulfills bounded differences with $c_{(i,j)} = 1$. For two $x, x^\prime \in \R^N$ with $d_H(x, x^\prime) \leq 1$, their $G_{\mu^p, \Sigma}(x)$ and $G_{\mu^p, \Sigma}(x^\prime)$ only differ in one edge and thus the property transfers to $x \mapsto \operatorname{MaxCut}(G_{\mu^p, \Sigma}(x))$. Hence, for all $t > 0$,
    \begin{align*}
        \P\Parent{\Abs{\operatorname{MaxCut}(G_{\mu^p, \Sigma}(X))- \E[\operatorname{MaxCut}(G_{\mu^p, \Sigma}(X))]} \geq t}
        \leq 2\exp\Parent{-\frac{2t^2}{\kappa^\star(\Sigma) \binom{n}{2}}}. 
    \end{align*}

    Here,  the inequality holds by Corollary \ref{cor:mcdiarmidsgaussian} and since $\twonorm{c}^2 = \binom{n}{2}$. 
\end{proof}

\subsection{DKW-type Inequality}

A fundamental question in empirical process theory is the convergence of the empirical cumulative distribution function (CDF) to its population counterpart. The Dvoretzky-Kiefer-Wolfowitz (DKW) inequality provides an answer to this in the case of \iid{} random variables where the rate of convergence is $1/\sqrt{n}$ (\cite{dvoretzky:kiefer:wolfowitz1956, massart1990}). In this application we consider the dependent setting and study the convergence of the empirical CDF towards its population version, the average marginal CDF as defined below. 

\begin{definition}{Cumulative Distribution Functions}
    \label{def:cdfs}

    Let $\nu$ be a probability measure on $\R^n$ and $X \sim \nu$. Respectively, define the empirical and average marginal cumulative distribution functions at $x \in \R$ as 
    \begin{align*}
        \hat F_n(x)
        := \frac{1}{n} \sum_{i=1}^n \indicator{X_i \leq x}, 
        \quad \text{and} \quad 
        \bar F(x)
        := \E[\hat F_n(x)]
        = \frac{1}{n} \sum_{i=1}^n \P(X_i \leq x). 
    \end{align*}

\end{definition}

\cite{bobkov:goetze2010} derived the following DKW-type inequality when the underlying joint measure $\nu$ fulfills a log-Sobolev inequality and the average marginal CDF is Lipschitz. 

\begin{theorem}{DKW-type Inequality under LSI (Theorem 1.2, \cite{bobkov:goetze2010})}
    \label{thm:DKWunderLSI}

    Let $X \sim \nu$ on $\R^n$ where $\nu$ is a probability measure fulfilling $\operatorname{LSI}(\rho)$. Assume that the average marginal cumulative distribution function $\bar F$ is $M$-Lipschitz. Then, for any $r > 0$,
    \begin{align*}
        \P\Parent{\sup_{x\in\R} |\bar F(x) - \hat F_n(x)| \geq r} 
        \leq \frac 4 r \exp\Parent{-\frac{2}{27}\frac{nr^3}{\rho M^2}}. 
    \end{align*}
\end{theorem}

We complement this result for measures fulfilling ATE through an application of Theorem \ref{thm:McDiarmidsATE} combined with a standard bracketing argument. 

\begin{theorem}{DKW-type Inequality under ATE}
    \label{thm:DKWunderATE}

    Let $X \sim \nu$ on $\R^n$ where $\nu$ is a probability measure fulfilling $\kappa$-ATE. Assume that the average marginal cumulative distribution function $\bar F$ is continuous. Then, for any $r > 0$,
    \begin{align*}
        \P\Parent{\sup_{x\in\R} |\bar F(x) - \hat F_n(x)| \geq r} 
        \leq \frac 4 r \exp\Parent{-\frac{nr^2}{2\kappa}}. 
    \end{align*}
\end{theorem}

\begin{remark}

    Theorem \ref{thm:DKWunderATE} improves Theorem \ref{thm:DKWunderLSI} in the dependence on $r^2$ versus $r^3$. Hence, it achieves the expected $1/\sqrt n$-rate of convergence up to logarithms as long as $\kappa \asymp 1$. While it is unclear which measures exactly fulfill ATE, the class of $\alpha$-strongly log-concave and $\beta$-log-smooth measures is an intersection between the ones fulfilling ATE and a log-Sobolev inequality (See Theorem \ref{thm:bakryemery}).
    
\end{remark}

\begin{proof}

    The proof combines our McDiarmid's inequality with a standard bracketing argument as can be found in the proof of \cite[Theorem 1.2]{bobkov:goetze2010} or \cite[Theorem 4.1.1]{vandegeer2020}. 

    For all $X, X^\prime \in \R^n$ \st{} $d_H(X,X^\prime) \leq 1$, let $k$ index the coordinate in which they differ. Define the corresponding empirical CDFs $\hat F_n(x)$ and $\hat F_n^\prime(x)$ for a fixed $x \in \R$. Then, we have 
    \begin{align*}
        \Abs{\hat F_n(x) - \hat F_n^\prime(x)}
        &= \Abs{\frac{1}{n} \sum_{i=1}^n \indicator{X_i \leq x} - \frac{1}{n} \sum_{i=1}^n \indicator{X_i^\prime \leq x}} \\
        &= \frac{1}{n} \Abs{\sum_{i=1}^n \indicator{X_i \leq x} -  \indicator{X_i^\prime \leq x}} \\
        &= \frac{1}{n} \Abs{\indicator{X_k \leq x} -  \indicator{X_k^\prime \leq x}}
        \leq \frac 1 n 
        =: c_k. 
    \end{align*}

    Thus, as a function of $X$, $\hat F_n(x)$ fulfills the bounded difference property with $\twonorm{c} = 1/n \cdot \twonorm{\one} = 1/\sqrt{n}$. Applying McDiarmid's inequality then yields that for all $t > 0$ we have
    \begin{align*}
        \P\Parent{\Abs{\hat F_n(x) - \bar F(x)} \geq t}
        = \P\Parent{\Abs{\hat F_n(x) - \E\bracket{\hat F_n(x)}} \geq t}
        \leq 2\exp\Parent{-\frac{2t^2}{\kappa \twonorm{c}^2}}
        = 2\exp\Parent{-\frac{2nt^2}{\kappa}}. 
    \end{align*}

    We translate this pointwise guarantee into a uniform one using bracketing. As $\bar F$ is continuous and increasing, for all $N \in \N$ we can choose $N+1$ anchor points
    \begin{align*}
        -\infty = a_0 < a_1 < ... < a_N = \infty
        \quad \text{\st{}} \quad 
        \bar F(a_j) - \bar F(a_{j-1}) = \frac 1 N
        \quad \text{for all $j \in [N]$}. 
    \end{align*}

    For all $x \in (a_{j-1}, a_j]$ we have $(-\infty, a_{j-1}] \subset (-\infty, x] \subset (-\infty, a_j]$ and thus it holds that 
    \begin{align*}
        \indicator{X_i \leq a_{j-1}}
        \leq \indicator{X_i \leq x}
        \leq \indicator{X_i \leq a_j} 
        \quad \text{for all $i \in [n]$}. 
    \end{align*}

    Through summing these over $i \in [n]$, for the empirical and average marginal CDF we have
    \begin{align*}
        \hat F_n(a_{j-1})
        \leq \hat F_n(x) 
        \leq \hat F_n(a_j)
        \quad \text{and} \quad 
        \bar F(a_{j-1}) 
        \leq \bar F(x) 
        \leq \bar F(a_j). 
    \end{align*}

    Thus, the following upper and lower bounds hold
    \begin{align*}
        \hat F_n(x) - \bar F(x)
        &\leq \hat F_n(a_j) - \bar F(x)
        \leq \hat F_n(a_j) - \bar F(a_{j-1})
        = \hat F_n(a_j) - \bar F(a_j) + 1/N,  \\
        \hat F_n(x) - \bar F(x)
        &\geq \hat F_n(a_{j-1}) - \bar F(x)
        \geq \hat F_n(a_{j-1}) - \bar F(a_j)
        = \hat F_n(a_{j-1}) - \bar F(a_{j-1}) - 1/N. 
    \end{align*}

    Taking absolute values and the maximum over $j \in [N]$ this yields that for $x \in R$, 
    \begin{align*}
        - \max_{j \in \curly{0,...,N}} \Abs{\hat F_n(a_j) - \bar F(a_j)} - \frac 1 N
        \leq \hat F_n(x) - \bar F(x) 
        \leq \max_{j \in \curly{0,...,N}} \Abs{\hat F_n(a_j) - \bar F(a_j)} + \frac 1 N. 
    \end{align*}

    Since $-C \leq y \leq C \Leftrightarrow \abs{y} \leq C$ and because the resulting \rhs{} is independent of $x$, 
    \begin{align*}
        \sup_{x \in \R} \Abs{\hat F_n(x) - \bar F(x) }
        \leq \max_{j \in \curly{0,...,N}} \Abs{\hat F_n(a_j) - \bar F(a_j)} + \frac 1 N. 
    \end{align*}

    Note here, that $\bar F(a_0) = \hat F_n(a_0) = 0$ and $\bar F(a_N) = \hat F_n(a_N) = 1$ and thus we only have to control $N-1$ terms in the maximum. It remains to do this using a union bound and the pointwise inequality: 
    \begin{align*}
        \P\Parent{\max_{j \in \curly{0,...,N}} \Abs{\hat F_n(a_j) - \bar F(a_j)} \geq t}
        \leq 2 (N-1) \exp\Parent{-\frac{2nt^2}{\kappa}}
        \leq \frac 4 r \exp\Parent{-\frac{2nt^2}{\kappa}}. 
    \end{align*}

    Choosing $N = \lceil \frac 2 r \rceil$ for which $\frac 2 r \leq N \leq 1 + \frac 2 r$ and $t = r/2$ then recovers the statement. 
\end{proof}

\begin{remark}

    For Gaussian measures the derivative of the average marginal CDF is
    \begin{align*}
        \frac{d}{dx} \bar F(x) 
        = \frac 1 n \sum_{i=1}^n \frac{d}{dx} \P(X_i \leq x) 
        = \frac 1 n \sum_{i=1}^n \frac{1}{\sqrt{2\pi\Sigma_{ii}}} \exp\Parent{-\frac{(x-\mu_i)^2}{2\Sigma_{ii}}}
        \leq \frac{1}{\sqrt{2\pi}}\frac 1 n \sum_{i=1}^n \frac{1}{\sqrt{\Sigma_{ii}}}
        =: M. 
    \end{align*}

    Both $\kappa^\star(\Sigma)$ as in Lemma \ref{lem:ATEtightgauss} and $\rho M^2 = M^2\opnorm{\Sigma}$ are upper bounded by $\kappa(\Sigma)$. However, they are generally incomparable. If $\Sigma$ is poorly conditioned due to high correlations $\kappa^\star(\Sigma)$ blows up whereas $\rho M^2$ is unaffected by this. In turn, if $\Sigma$ is poorly conditioned due to scale imbalances across dimensions, $\rho M^2$ will explode whereas $\kappa^\star(\Sigma)$ remains bounded. 
        
\end{remark}

\section{Conclusion}
\label{sec:discussion}

While the concentration of measure phenomena underlying the presented McDiarmid's-type concentration inequalities under dependence play an important role in mixing time analyses for Gibbs samplers, they have not yet been exploited more broadly in mathematical statistics, learning theory and theoretical computer science. We believe that this is, because they have not yet been given sufficient attention in the form of McDiarmid's and thus at a level of abstraction that is suitable for such problems as showcased by our applications. The applications we tackle are thereby only a first step in investigating the influence of dependence on concentration and convergence rates in a broad class of problems that exhibit Gaussian-like dependence and fulfill bounded differences. By analogy to Gaussian Lipschitz concentration, we expect the Gaussian version of McDiarmid's -- Gaussian Hamming-Lipschitz concentration -- to be fruitful in many domains and therefore invite researchers to apply it to problems in their specific fields. We think that an iterative back and forth between tools and applications is the most effective approach to extending theory usually built on \iid{} assumptions to (weakly) dependent cases. The DKW-type inequality derived in Theorem \ref{thm:DKWunderATE} serves as an example for this. On the one hand, it shows that McDiarmid's under ATE is a useful tool, which combined with standard arguments can improve upon existing theory and establish the expected $1/\sqrt{n}$-convergence rate. On the other hand, the comparison to the DKW-type inequality under LSI in Theorem \ref{thm:DKWunderLSI} due to \cite{bobkov:goetze2010} raises the question why the constants $\kappa^\star$ and $\rho M^2$ capturing the dependence in both cases may behave so differently. Future work should thus investigate whether this difference is merely a proof artifact or if there is an underlying more powerful characterization of dependence at work. One that ATE might not capture. 

\section{Appendix}

\subsection{Deep Dives into Related Work}

In this subsection we instantiate the foundational work by \cite{marton2013} to show that her results yield an ATE for Gaussian measures, albeit under more restrictive conditions on the covariance matrix $\Sigma$ than the results in Section \ref{sec:ATE}. Moreover, we outline when the McDiarmid's-type inequality in \cite{esposito:mondelli2023} yields dimension-free Gaussian concentration. 

\subsubsection{Marton's ATE}

\begin{definition}{Block-Interaction Matrix}

    Let $\nu$ be a probability measure on $\R^n$ with density $\nu(x) \propto \exp(-V(x))$. Let $x, y \in \R^n$ and $z^{(k)}(x,y) := (x_{I_1},..., y_{I_k},..., x_{I_m})\T$. Define the matrix A via its entries for $k, \ell \in [m]$ and $i \in I_k, j \in I_\ell$
    \begin{align*}
        A_{ij}^\rho(x,y)
        = \frac{1}{\sqrt{\rho_k-\rho} \cdot \sqrt{\rho_\ell-\rho}} \cdot 
        \begin{cases}
            \nabla_x^2 V(z^{(k)})_{ij} & \text{if $k \neq \ell$}, \\
            0 & \text{if $k = \ell$}.
        \end{cases}
    \end{align*}
    
\end{definition}

\begin{assumption}{Marton's Necessary Condition}
    \label{ass:marton}

    A probability measure $\mu$ with density $p(x) \propto \exp(-V(x))$ fulfills this assumption if 
    \begin{enumerate}
        \item The conditional measures $\nu(\cdot| x_{-I_k})$ are $\operatorname{LSI}(\rho_k)$ for all $k \in [n]$, 
        \item The matrix $[\nabla_x^2 V(x)]_{I_k} \succeq \gamma I_{\cardinality{I_k}}$ for $\gamma \in \R$ for all $k \in [n]$, 
        \item For $0 < \rho < \min_{k \in [m]} \rho_k$ we have $\sup_{x,y \in \R^n} \opnorm{A^0(x,y)} < 1$ and $\sup_{x,y \in \R^n} \opnorm{A^\rho(x,y)} \leq 1$. 
    \end{enumerate}

\end{assumption}

\begin{theorem}{Marton's ATE (Theorem 1 \cite{marton2013})}
    \label{thm:ATEmartons}

    Let $\mu$ and $\nu$ be probability measures on $\Omega^n$. Suppose $\mu$ fulfills the conditions in Assumption \ref{ass:marton} with constants $\rho_1,...,\rho_m, \rho > 0$. Then, the following ATE holds
    \begin{align*}
        \operatorname{D_{KL}}(\nu||\mu)
        \leq \sum_{k=1}^n \frac{\rho_k}{\rho} \cdot \nu\Bracket{\operatorname{D_{KL}}\Parent{\nu(\cdot | X_{-I_k}) || \mu(\cdot|X_{-I_k})}}. 
    \end{align*}
\end{theorem}

\begin{corollary}{Marton's ATE for Gaussians}

    Let $\mathcal N$ on $\R^n$ be a multivariate Gaussian with mean $\mu \in \R^n$ and $\Sigma \in \R^{n \times n}$. Further, let $\rho_k = (\Sigma\inv)_{kk}$ and $\rho \in (0, \min_{k \in [n]}\rho_k)$ be the biggest $\rho$ \st{} $\rho I_n \preceq \Sigma\inv \preceq 2\diagonal{\Sigma\inv} - \rho I_n$. Then, we have that
    \begin{align*}
        \ent[\mathcal N]{f}
        \leq \sum_{i=1}^n \frac{\rho_i}{\rho} \cdot \E[\ent[\mathcal N(\cdot | X_{-i})]{f}].
    \end{align*}

\end{corollary}

\begin{proof}

    We start by verifying Assumption \ref{ass:marton}: 
    \begin{enumerate}
        \item For all $k \in [n]$, the conditional density of $\nu$ given $x_{-I_k} = x_{-k}$ is
        \begin{align*}
            \mathcal N(x_k| x_{-k})
            \propto \exp\Parent{-\frac 1 2 x\T P x}
            \propto \exp\Parent{-\frac 1 2 x_k P_{kk} x_k - x_k\T [P]_{k,-k}x_{-k}}.
        \end{align*}

        We have $-\nabla_{x_k}^2 \log(\mathcal N(x_k|x_{-k})) \succeq P_{kk} =: \rho_k$. Therefore, by the Bakry-Emery criterion of Theorem \ref{thm:bakryemery} the measures $\mathcal N(\cdot | x_{-k})$ are $\operatorname{LSI}(\rho_k)$. 

        \item For $\mu$ we have $\nabla_x^2 V(x) = \Sigma\inv$ and therefore $[\nabla_x^2 V(x)]_k = P_{kk} \succeq \rho_k$.

        \item Since $\nabla_x^2 V(x) = \Sigma\inv$ is constant in $x$, the block interaction matrix is defined via the entries 
        \begin{align*}
            A_{ij}^\rho(x,y)
            = \frac{1}{\sqrt{\rho_i-\rho} \cdot \sqrt{\rho_j-\rho}} \cdot 
            \begin{cases}
                P_{ij} & \text{if $i \neq j$}, \\
                0 & \text{if $i = j$}.
            \end{cases}
        \end{align*}
    
        Note that $\rho_k = P_{kk}$ and hence defining $D := \diagonal{P}$ this simplifies to
        \begin{align*}
            A^\rho = (D - \rho I_n)^{-1/2} (P-D)(D - \rho I_n)^{-1/2}. 
        \end{align*}
    
        Now, the condition $\opnorm{A^0} < 1$ is equivalent to $-I_n \prec A^0 \prec I_n$ which simplifies as 
        \begin{align*}
            -I_n \prec D^{-1/2} (P-D)D^{-1/2} \prec I_n
            \; \Leftrightarrow \; 
            0 \prec P \prec 2D. 
        \end{align*}
    
        Further, the condition $\opnorm{A^\rho} \leq 1$ is equivalent to $-I_n \preceq A^\rho \preceq I_n$ which simplifies as 
        \begin{align*}
            -I_n \preceq (D - \rho I_n)^{-1/2} (P-D)(D - \rho I_n)^{-1/2} \preceq I_n
            &\; \Leftrightarrow \; 
            \rho I_n - D \preceq P-D \preceq D-\rho I_n \\
            &\; \Leftrightarrow \; 
            \rho I_n \preceq P \preceq 2D-\rho I_n. 
        \end{align*}

        This latter condition holds by assumption and as $\rho > 0$, we have $0 \prec P \prec 2D \Leftrightarrow \opnorm{A^0} < 1$. 
    \end{enumerate}

    An application of Theorem \ref{thm:ATEmartons} and Auxiliary Result \ref{aux:equinotionsATE} yield the statement. 
\end{proof}

\subsubsection{McDiarmid's via Hellinger Integrals}

\begin{definition}{Hellinger Integral of Order $\alpha$}

    Let $\mu, \nu$ be two probability measures \st{} $\mu \ll \nu$. We define the Hellinger integral of order $\alpha$ as 
    \begin{align*}
        H_\alpha (\mu || \nu)
        := \int \Parent{\frac{d\mu}{d\nu}}^\alpha d\nu. 
    \end{align*}    
\end{definition}

\begin{theorem}{McDiarmid's via Hellinger Integrals (Theorem 1, \cite{esposito:mondelli2023})}
    \label{thm:mcdiarmidshellinger}

    Let $\nu$ on $\Omega^n$ be a probability measure with marginals $\nu_i$ on $\Omega$. Let $\nu \ll \otimes_{i=1}^n \nu_i$ and $f : \Omega \to \R$ fulfill the bounded differences property with constants $c := (c_1,...,c_n)\T$. Then, for $t > 0$ and $\alpha > 1$, 
    \begin{align*}
        \P\Parent{\Abs{f(X) - \int_\Omega f(x) d\otimes_{i=1}^n \nu_i(x)} \geq t}
        \leq 2^{\frac{\alpha}{\alpha - 1}} \exp\Parent{- \frac{2t^2}{\frac{\alpha}{\alpha - 1} \twonorm{c}^2}} H_\alpha(\nu|| \otimes_{i=1}^n \nu_i)^{\frac 1 \alpha}. 
    \end{align*}
\end{theorem}

\begin{lemma}{Hellinger Integrals for Zero-Mean Gaussians}
    \label{lem:hellintgauss}
    
    Let $\mathcal N_1$ and $\mathcal N_2$ on $\R^n$ be zero-mean Gaussians with covariances matrices $\Sigma_1, \Sigma_2$. Assume that $M := \alpha \Sigma_1\inv + (1-\alpha) \Sigma_2\inv \succeq 0$. Then, the Hellinger integrals of order $\alpha > 1$ have the form
    \begin{align*}
        H_\alpha(\mathcal N_1 || \mathcal N_2)
        = \det(\Sigma_1)^{-\alpha/2} \det(\Sigma_2)^{-(1-\alpha)/2} \det(M)^{-1/2}. 
    \end{align*}
    
\end{lemma}

\begin{proof}

    Let $p$ and $q$ be the densities of $\mathcal N_1$ and $\mathcal N_2$ with respect to Lebesgue measure. Then, 
    \begin{align*}
        H_\alpha(\mathcal N_1 || \mathcal N_2)
        &= \int_{\R^n} \Parent{\frac{d\mu}{d\nu}}^\alpha d\nu 
        = \int_{\R^n} p(x)^\alpha q(x)^{1-\alpha} dx
        = \int_{\R^n} \mathcal N(x; 0, \Sigma_1)^\alpha \mathcal N(x; 0, \Sigma_2)^{1-\alpha} dx \\
        &= (2\pi)^{-n/2} \det(\Sigma_1)^{-\alpha/2} \det(\Sigma_2)^{-(1-\alpha)/2} \int_{\R^n} \exp\Parent{-\frac 1 2 x\T \Parent{\alpha \Sigma_1\inv + (1-\alpha) \Sigma_2\inv } x} dx \\
        &= (2\pi)^{-n/2} \det(\Sigma_1)^{-\alpha/2} \det(\Sigma_2)^{-(1-\alpha)/2} \cdot Z_M \int_{\R^n} \frac{1}{Z_M}\exp\Parent{-\frac 1 2 x\T M x} dx \\
        &= (2\pi)^{-n/2} \det(\Sigma_1)^{-\alpha/2} \det(\Sigma_2)^{-(1-\alpha)/2} \cdot (2\pi)^{n/2} \det(M\inv)^{1/2} \\
        &= \det(\Sigma_1)^{-\alpha/2} \det(\Sigma_2)^{-(1-\alpha)/2} \det(M\inv)^{1/2}.
    \end{align*}

    The statement is recovered as $\det(M\inv) = \det(M)\inv$. 
\end{proof}

\begin{corollary}{Hellinger Integral between Joint Gaussian and Marginals}

    Let $\mathcal N$ on $\R^n$ be a multivariate Gaussian with zero-mean and covariance $\Sigma \in \R^{n \times n}$ and let $\mathcal N_i$ be its $i$-th marginal. Let $X \sim \mathcal N$ and $\cor[\mathcal N]{X} := D^{-1/2} \Sigma D^{-1/2}$ with $D := \diagonal{\Sigma}$ and let $\lambda_1,...,\lambda_n$ be the eigenvalues of $\cor[\mathcal N]{X}$. For all $\alpha > 1$, under the condition that $\cor[\mathcal N]{X} \preceq \frac{\alpha}{\alpha-1} \cdot I_n$ we have 
    \begin{align*}
        H_\alpha(\mathcal N || \otimes_{i=1}^n \mathcal N_i) 
        = \prod_{i=1}^n (\alpha \lambda_i^{\alpha-1} + (1-\alpha) \lambda_i^{\alpha})^{-1/2}.
    \end{align*}

\end{corollary}

\begin{remark}

    Taking a logarithm simplifies the analysis of the \rhs{} above: 
    \begin{align*}
        \log\Parent{H_\alpha(\mathcal N || \otimes_{i=1}^n \mathcal N_i)} 
        = - \frac 1 2 \sum_{i=1}^n \log(\alpha \lambda_i^{\alpha-1} + (1-\alpha) \lambda_i^{\alpha}). 
    \end{align*}

    By Taylor expanding the summands around $\lambda_i = 1$, the case of independence, we may see that under suitable conditions on the eigenvalues, the Hellinger integral can be controlled by $\frobnorm{\cor[\mathcal N]{X} - I_n}$. Hence, Theorem \ref{thm:mcdiarmidshellinger} may also yield dimension-free concentration under weak dependence. However, the control of dependence is much less explicit than in Corollary \ref{cor:mcdiarmidsgaussian} and the concentration is around the expectation of $f(X)$ \wrt{} the product measure $\otimes_{i=1}^n \mathcal N_i$. 
    
\end{remark}

\begin{proof}

    Under the condition that $M := \alpha \Sigma\inv + (1-\alpha) I_n \succeq 0$ Lemma \ref{lem:hellintgauss} yields 
    \begin{align*}
        H_\alpha(\mathcal N || \otimes_{i=1}^n \mathcal N_i) 
        &= \det(\Sigma)^{-\alpha/2} \det(D)^{-(1-\alpha)/2} \det(\alpha \Sigma\inv + (1-\alpha) D\inv)^{-1/2} \\
        &= \det(C)^{-\alpha/2} \det(D)^{-\alpha/2} \det(D)^{-(1-\alpha)/2} \det(D^{-1/2}(\alpha C\inv + (1-\alpha) I_n)D^{-1/2})^{-1/2} \\
        &= \det(C)^{-\alpha/2} \det(\alpha C\inv + (1-\alpha) I_n)^{-1/2} \\
        &= \det(\alpha C^{\alpha-1} + (1-\alpha) C^\alpha)^{-1/2}. 
    \end{align*}

    Above, we use that $\det(\Sigma) = \det(C)\det(D)$ and $\det(ABA) = \det(A)^2\det(B)$ for $A, B \in \R^{d \times d}$. By the Spectral Mapping Theorem, any matrix polynomial $p(C)$ has eigenvalues $p(\lambda_1),...,p(\lambda_2)$. Since the determinant is just the product of the eigenvalues of a matrix, we thus have
    \begin{align*}
        H_\alpha(\mathcal N || \otimes_{i=1}^n \mathcal N_i) 
        = \prod_{i=1}^n (\alpha \lambda_i^{\alpha-1} + (1-\alpha) \lambda_i^{\alpha})^{-1/2}.
    \end{align*}

    Using that $A \preceq B \Leftrightarrow B\inv \preceq A\inv$ for $A, B \in \R^{d \times d}$ and by left- and right-multiplication of $D^{-1/2}$: 
    \begin{align*}
        M = \alpha \Sigma\inv + (1-\alpha)D\inv \succeq 0
        \; &\Leftrightarrow \; 
        \alpha \Sigma\inv \succeq (\alpha-1)D\inv \\
        \; \Leftrightarrow \; 
        \Sigma \preceq \frac{\alpha}{\alpha-1} D
        \; &\Leftrightarrow \; 
        D^{-1/2} \Sigma D^{-1/2} \preceq \frac{\alpha}{\alpha-1}  I_n. 
    \end{align*}

    The statement is then recovered as $\cor{\mu} = D^{-1/2} \Sigma D^{-1/2}$.  
\end{proof}

\begin{remark}

    $\cor[\mathcal N]{X} \preceq \frac{\alpha}{\alpha-1} \cdot I_n$ is equivalent to $\Sigma \preceq \frac{\alpha}{\alpha-1} \cdot \diagonal{\Sigma}$. 
    
\end{remark}

\subsection{Auxiliary Results}

\begin{theorem}{Bakry-Emery (Theorem 21.2, \citet{villani2009}, Corollary 5.7.2, \citet{bakry:gentil:ledoux2014})}
    \label{thm:bakryemery}

    Let the probability measure $\nu$ on $\R^n$ have density $\nu(x) \propto \exp(-V(x))$ for all $x \in \R^n$. Suppose there is $\alpha > 0$ \st{} $\nabla^2 V(x) \succeq \alpha I_n$ for all $x \in \R^n$. Then, the following holds: 
    \begin{align*}
        \ent[\nu]{f^2} 
        \leq \frac 2 \alpha \cdot \E[\twonorm{\nabla f(X)}^2]
        \quad \text{and} \quad 
        \var[\nu]{f} 
        \leq \frac 2 \alpha \cdot \E[\twonorm{\nabla f(X)}^2].
    \end{align*}
\end{theorem}

\begin{auxiliary}{Homogeneity of Entropy}
    \label{aux:homoent}

    Let $X \sim \nu$ where $\nu$ is a probability measure on $\R^n$. For $f : \R^n \to [0, \infty)$ and $\alpha > 0$ we  have 
    \begin{align*}
        \ent[\nu]{\alpha f} = \alpha \ent[\nu]{f}.
    \end{align*}

\end{auxiliary}

\begin{proof}

    By the definition of entropy we have 
    \begin{align*}
        \ent[\nu]{\alpha f}
        = \E\Bracket{\alpha f(X) \log\Parent{\frac{\alpha f(X)}{\E[\alpha f(X)]}}}
        = \alpha \E\Bracket{f(X) \log\Parent{\frac{f(X)}{\E[f(X)]}}}
        = \alpha \ent[\nu]{f}. 
    \end{align*}
\end{proof}

\begin{auxiliary}{Equivalent Notions of ATE (Definition 30, \cite{anari:koehler:vuong2024})}
    \label{aux:equinotionsATE}

    Let $\mu$ on $\Omega^n$ be a probability measure fulfilling ATE. Then, the following two notions are equivalent: 
    \begin{align}
        \DKL{\nu}{\mu}
        &\leq C \sum_{i=1}^n \E[\DKL{\nu(\cdot|X_{-i})}{\mu(\cdot|X_{-i})}] & \text{for all $\nu$ on $\Omega^n$ \st{} $\nu \ll \mu$}, \label{eq:ATEKL} \\
        \ent[\mu]{f}
        &\leq C\sum_{i=1}^n \E[\ent[\mu(\cdot | X_{-i})]{f}]  & \text{for all $f : \Omega^n \to (0, \infty)$}. \label{eqATEent}
    \end{align}
\end{auxiliary}

\begin{proof}
    In this proof we use the notation $\nu[f] := \int f d\nu$ to make it more explicit \wrt{} which measure an expectation is taken. Recall that the Kullback-Leibler divergence is $\DKL{\nu}{\mu} := \ent[\nu]{\frac{d\nu}{d\mu}}$. 
    \begin{enumerate}
        \item \eqref{eq:ATEKL} $\Rightarrow$ \eqref{eqATEent}: Let $f : \Omega^n \to (0, \infty)$ be arbitrary. If $f = 0$ the statement in \eqref{eqATEent} holds trivially. For $f > 0$, define a measure $\nu$ via the Radon-Nikodym derivative $d\nu/d\mu = f/\mu[f]$. The measure $\nu$ is a probability measure, because $\mu[d\nu/d\mu] = \mu[f/\mu[f]] = 1$. By definition, it holds that
        \begin{align*}
            \DKL{\nu}{\mu} 
            = \ent[\mu]{\frac{d\nu}{d\mu}} = \mu\Bracket{\frac{d\nu}{d\mu}\log\Parent{\frac{d\nu}{d\mu}}}
            = \mu\Bracket{\frac{f}{\mu[f]} \log\Parent{\frac{f}{\mu[f]}}}
            = \frac{1}{\mu[f]} \ent[\mu]{f}. 
        \end{align*}        
        
        By definition of the regular conditional probability, 
        \begin{align*}
            \frac{d\nu(x_i|x_{-i})}{d\mu(\cdot|x_{-i})}
            := \frac{\frac{d\nu(x)}{d\mu}}{\mu[\frac{d\nu}{d\mu}|x_{-i}]}
            = \frac{f(x)/\mu[f]}{\mu[f/\mu[f]|x_{-i}]}
            = \frac{f(x)}{\mu[f|x_{-i}]}. 
        \end{align*}

        By the same reasoning as for $d\nu/d\mu$ above, we have
        \begin{align*}
            \DKL{\nu(\cdot|x_{-i})}{\mu(\cdot|x_{-i})}
            &= \mu\Bracket{\frac{d\nu(x_i|x_{-i})}{d\mu(\cdot|x_{-i})} \log\Parent{\frac{f}{\mu[f|x_{-i}]}}\Big|x_{-i}}
            = \nu\Bracket{\log\Parent{\frac{f}{\mu[f|x_{-i}]}}\Big|x_{-i}}. 
        \end{align*}

        Integrating \wrt{} the measure $\nu$ this becomes
        \begin{align*}
            \nu[\DKL{\nu(\cdot|X_{-i})}{\mu(\cdot|X_{-i})}]
            &= \nu\Bracket{\nu\Bracket{\log\Parent{\frac{f}{\mu[f|x_{-i}]}}\Big|X_{-i}}}
            = \nu\Bracket{\log\Parent{\frac{f}{\mu[f|x_{-i}]}}} \\
            &= \mu\Bracket{\frac{d\nu}{d\mu}\log\Parent{\frac{f}{\mu[f|x_{-i}]}}} 
            = \frac{1}{\mu[f]} \mu\Bracket{f\log\Parent{\frac{f}{\mu[f|x_{-i}]}}} \\
            &= \frac{1}{\mu[f]} \mu\Bracket{\mu\Bracket{f\log\Parent{\frac{f}{\mu[f|x_{-i}]}}\Big| X_{-i}}}
            = \frac{1}{\mu[f]} \mu\Bracket{\ent[\mu(\cdot|X_{-i})]{f}}.
        \end{align*}

        Substituting the expressions into \eqref{eq:ATEKL} and multiplying the inequality by $\mu[f]$ recovers \eqref{eqATEent}. 

        \item \eqref{eqATEent} $\Rightarrow$ \eqref{eq:ATEKL}: Let $\nu$ on $\Omega^n$ be an arbitrary probability measure \st{} $\nu \ll \mu$. Define $f = d\nu/d\mu$ and note that since $\nu$ is a probability measure, we have $\mu[f] = \mu[d\nu/d\mu] = 1$. Hence, by the above 
        \begin{align*}
            \ent[\mu]{f}
            = \ent[\mu]{\frac{d\nu}{d\mu}}
            = \DKL{\nu}{\mu}, 
            \quad
            \mu\Bracket{\ent[\mu(\cdot|X_{-i})]{f}}
            = \nu[\DKL{\nu(\cdot|X_{-i})}{\mu(\cdot|X_{-i})}]. 
        \end{align*}

        Substituting the expressions into \eqref{eqATEent} recovers \eqref{eq:ATEKL}. 
    \end{enumerate}
\end{proof}

\printbibliography

\end{document}